\documentclass[12pt,reqno]{amsart}

\usepackage{amsthm, mathrsfs,amssymb,amsmath}
\usepackage{enumerate}
\usepackage[hidelinks]{hyperref}
\usepackage{xcolor}
\hypersetup{
	colorlinks,
	linkcolor={red!50!black},
	citecolor={green!50!black},
	urlcolor={red!80!black}
}
\makeatletter
\@namedef{subjclassname@2010}{%
	\textup{2010} Mathematics Subject Classification}
\makeatother

\allowdisplaybreaks

\newtheorem{thm}{Theorem}[section]
\newtheorem{lem}[thm]{Lemma}

\newtheorem{cor}[thm]{Corollary}

\theoremstyle{definition}
\theoremstyle{remark}

\title[$K$-functional and Modulus of smoothness on Damek-Ricci spaces]{A note on  $K$-functional, Modulus of smoothness, Jackson theorem and Nikolskii-Stechkin inequality on Damek-Ricci spaces }

\author{Vishvesh Kumar}
\address{Vishvesh Kumar \endgraf
	Department of Mathematics: Analysis, Logic and Discrete Mathematics
	\endgraf
	Ghent University, Belgium}
\endgraf
\email{vishveshmishra@gmail.com}
\author[Michael Ruzhansky]{Michael Ruzhansky}
\address{
	Michael Ruzhansky
	\endgraf
	Department of Mathematics: Analysis, Logic and Discrete Mathematics
	\endgraf
	Ghent University, Belgium
	\endgraf
	and
	\endgraf
	School of Mathematics
	\endgraf
	Queen Mary University of London
	\endgraf
	United Kingdom
	\endgraf
	{\it E-mail address} {\rm michael.ruzhansky@ugent.be}
}

\begin{document}
	
	\begin{abstract} 
	 In this paper we study approximation theorems for $L^2$-space on Damek-Ricci spaces. We prove direct Jackson theorem of approximations for the modulus of smoothness defined using spherical mean operator on Damek-Ricci spaces. We also prove Nikolskii-Stechkin inequality. To prove these inequalities we use functions of bounded spectrum as a tool of approximation. Finally, as an application we prove equivalence of the $K$-functional and modulus of smoothness for Damek-Ricci spaces.   
	
	\end{abstract}
	\maketitle
	\tableofcontents 
	
	\section{Introduction}  \label{Sec1} 
	The main purpose of this paper is to study the equivalence of the $K$-functional and the modulus of smoothness generated by the spherical mean operator on Damek-Ricci spaces. Damek-Ricci spaces, also known as Harmonic NA groups, are solvable (non-unimodular) Lie groups. It is worth mentioning that Damek-Ricci spaces contain non-compact symmetric spaces of rank one as a very small subclass and, in general, Damek-Ricci spaces are not symmetric. Damek-Ricci spaces were introduced by Eva Damek and Fulvo Ricci in \cite{DamekRicci} and the geometry of these spaces was  studied by Damek \cite{Damek} and Cowling-Dooley-Koranyi \cite{CDK91}. Fourier analysis on these spaces has been developed and studied by many authors including Anker-Damek-Yacoub \cite{Anker96}, Astengo-Comporesi-Di Blasio \cite{ACB97} , Damek-Ricci \cite{DamekRicci1}, Di Blasio \cite{Di-Blasio}, Ray-Sarkar \cite{RS}, Kumar-Ray-Sarkar \cite{KRS}. One of the interesting features of these spaces is that the radial analysis on these spaces behaves similar to the hyperbolic spaces as observed in \cite{Anker96} and therefore it fits into the perfect setting of Jacobi analysis developed by Flensted-Jensen and Koornwinder \cite{Koorn,FK,FK2}. 
	
	The study of the $K$-functional is a classical and important topic in interpolation theory and approximation theory. Peetre the $K$-functional is useful for describing the interpolation spaces between two Banach spaces.  First, let us recall the definition of the $K$-functional. For two Banach spaces $A_1$ and $A_2$, the Peetre the $K$-functional is given by
	$$K(f, \delta, A_1, A_2):= \inf \{\|f_1\|_{A_1}+\delta \|f_2\|_{A_2}:\,\,f=f_1+f_2, \, f_1 \in A_1, f_2 \in A_2  \},$$ where $\delta$ is a positive parameter. 
	Now, the Peetre interpolation space $(A_1, A_2)_{\theta, r}$ for $0<\theta<1,\, 0<r \leq \infty,$ is defined by the norm 
	$$|f|_{(A_1, A_2)_{\theta, r}}:= \begin{cases} \left( \int_0^\infty  [\delta^{-\theta} K(f, \delta, A_1, A_2)]^{r}  \frac{d \delta}{\delta} \right)^{\frac{1}{r}} & \quad \text{if} \,\, 0<r<\infty, \\ \sup_{\delta>0} \delta^{-\theta} K(f, \delta, A_1, A_2) & \quad \text{if}\,\, r=\infty.
	\end{cases}$$
	
	The characterizations of the $K$-functional has several applications in approximation theory \cite{Dit1}.  In \cite{Peetre}, Peetre started characterization of the $K$-functional by proving an   equivalence of it with the modulus of smoothness for $L^p$-spaces on $\mathbb{R}^n$ which proved to be very helpful to study apporximation theory.  Later, in \cite{Vore2} the authors showed its equivalence in terms of the rearrangement of derivatives for a pair of Sobolev spaces $W_p^m$ and for the pair $(L^p, W_p^m).$ In particular, a characterization of the $K$-functional for $(L^2(\mathbb{R}), W^m_2(\mathbb{R}))$ can be found in the classical book of Berens and Buter \cite{But}.
	 The characterizations of the $K$-functional for  the pair $(L^2(X), W^m_2(X))$ were explored by several authors for different choices of $X.$ Classically, this equivalence was proved for $X=\mathbb{R}^n$ by Peetre \cite{Peetre} and after that it was proved for $X=[a,b]$ by De Vore-Scherer \cite{Vore2}, for weighted setting by Ditzian \cite{Dit}, for $X=\mathbb{R}^n$ with Dunkl translation by Belkina and Platonov \cite{BP}, for rank one symmetric spaces \cite{Ouadih}, for Jacobi analysis in  \cite{HD20} and for compact symmetric spaces on \cite{Platov7}. In this paper our aim is to extend this characterization to more general setting of solvable (non unimodular) Lie groups. We consider the pair $(L^2(X), W^m_2(X))$ for $X$ being the Damek-Ricci spaces. We will prove the equivalence of the $K$-functional and modulus of smoothness generated by spherical mean operator on Damek-Ricci space. Modulus of smoothness for Damek-Ricci space has been introduced in \cite{KRS}. We prove our main result by establishing two classical results, namely, Direct Jackson theorem \cite{Niko1} and Nikolskii-Stechkin inequality \cite{Niko} for Damek-Ricci spaces. Platonov studied Direct Jackson theorem and Nikolskii-Stechkin inequality for compact homogeneous manifolds and for noncompact symmetric spaces of rank one (\cite{Platov3, Platov4,Platov5, Platov7,Platov8, Platov9}).

		\section{ Essentials about harmonic $NA$ groups} \label{Ess}
	
	For basics of harmonic $NA$ groups and Fourier analysis on them, one can refer to seminal research papers \cite{Damek,DamekRicci,DamekRicci1, Di-Blasio,Anker96,ACB97,CDK91, RS, KRS, Kaplan}. However, we give necessary definitions, notation and  terminology that we shall use in this paper.
	
	 Let $\mathfrak{n}$ be a two-step nilpotent Lie algebra, equipped with an inner product $\langle \,,\, \rangle$ . Denote by $\mathfrak{z}$ the center of $\mathfrak{n}$ and by $\mathfrak{v}$ the orthogonal complement of $\mathfrak{z}$ in $\mathfrak{n}$ with respect to the inner product of $\mathfrak{n}.$  
	We assume that dimensions of $\mathfrak{v}$ and $\mathfrak{z}$ are $m$ and $l$ respectively as real vector spaces. The Lie algebra $\mathfrak{n}$ is $H$-type algebra if for every $Z \in \mathfrak{z},$ the map $J_Z:\mathfrak{v} \rightarrow \mathfrak{v}$  defined by 
	$$\langle J_Z X, Y\rangle = \langle Z, [X, Y] \rangle,\,\,\,\,\,\,X,Y \in \mathfrak{v},\, Z\in \mathfrak{z},$$ satisfies the condition $J_Z^2=-\|Z\|^2I_{\mathfrak{v}},$ where $I_{\mathfrak{v}}$ is the identity operator on $\mathfrak{v}.$ It follows that for $Z \in \mathfrak{z}$ with $\|Z\|=1$ one has $J_Z^2=-I_{\mathfrak{v}}$; that is, $J_Z$ induced a complex structure on $\mathfrak{v}$ and hence $m=\dim(\mathfrak{v})$ is always even.  A connected and simply connected Lie group $N$ is called $H$-type if its Lie algebra is of $H$-type. The exponential map is a diffeomorphism as $N$ is nilpotent, we can parametrize the element of $N=\exp{\mathfrak{n}}$ by $(X, Z)$, for $X \in \mathfrak{v}$  and $Z \in \mathfrak{z}.$ The multiplication on $N$ follows from the Campbell-Baker-Hausdorff formula given by 
	$$(X, Z) (Z',Z')= (X+X', Z+Z'+\frac{1}{2}[X,X']).$$
	The group $A= \mathbb{R}_+^*$ acts on $N$ by nonisotropic dilations as follows: $(X, Y) \mapsto (a^{\frac{1}{2}}X, aZ).$ Let $S=N \ltimes A$ be the semidirect product of $N$ with $A$ under the aforementioned action. The group multiplication on $S$ is defined by 
	$$(X,Z,a)(X',Z',a')= (X+a^{\frac{1}{2}}X', Z+aZ'+\frac{1}{2}a^{\frac{1}{2}} [X, X'], aa').$$
	Then $S$ is a solvable (connected and simply connected) Lie group with Lie algebra $\mathfrak{s}=\mathfrak{z}\oplus \mathfrak{v} \oplus \mathbb{R}$ and Lie bracket 
	$$[(X, Z, \ell), (X', Z', \ell')]= (\frac{1}{2} \ell X'-\frac{1}{2} \ell'X, \ell Z'-\ell'Z+[X, X]', 0).$$
	The group $S$ is equipped with the left-invariant Riemannian metric induced by 
	$$ \langle (X, Z, \ell), (X', Z', \ell')\rangle= \langle X, X'\rangle+\langle Z, Z'\rangle+\ell \ell'$$ on $\mathfrak{s}.$ The homogneous dimension of $N$ is equal to $\frac{m}{2}+l$ and will be denoted by $Q.$ At times, we also use symbol $\rho$ for $\frac{Q}{2}.$ Hence $\dim(\mathfrak{s})=m+l+1,$ denoted by $d.$ The associated left Haar measure on $S$ is given by $a^{-Q-1} dX dZ da,$ where $dX,\, dZ$ and $da$ are the Lebesgue measures on $\mathfrak{v}, \mathfrak{z}$ and $\mathbb{R}_+^*$ respectively. The element of $A$ will be identified with $a_t=e^t,$ $t \in \mathbb{R}.$ The group $S$ can be realized as the unit ball $B(\mathfrak{s})$ in $\mathfrak{s}$ using the Cayley transform $C: S \rightarrow  B(\mathfrak{s})$ (see \cite{Anker96}).
	
	To define (Helgason) Fourier transform on $S$ we need to introduce the notion of Poisson kernel \cite{ACB97}. The Poisson Kernel $\mathcal{P}:S \times N \rightarrow \mathbb{R}$ is defined by $\mathcal{P}(na_t, n')= P_{a_t}(n'^{-1}n),$ where $$ P_{a_t}(n)= P_{a_t}(X, Z)=C a_t^Q \left( \left(a_t+\frac{|X|^2}{4} \right)^2+|Z|^2 \right)^{-Q},\,\,\,\, n=(X, Z) \in N.$$
	The value of $C$ is suitably adjusted so that $\int_N P_a(n) dn=1$ and $P_1(n) \leq 1.$ The Poisson kernel satisfies several useful properties (see \cite{KRS,RS,ACB97}), we list here a few of them. For $\lambda \in \mathbb{C},$ the complex power of the Poisson kernel is defined as 
	$$\mathcal{P}_\lambda(x, n)= \mathcal{P}(x, n)^{\frac{1}{2}-\frac{i \lambda}{Q}}.$$ It is known  (\cite{RS, ACB97}) that for each fixed $x \in S,$ $\mathcal{P}_\lambda(x, \cdot) \in L^p(N)$ for $1 \leq p \leq \infty$ if $\lambda = i \gamma_p \rho,$ where $\gamma_p= \frac{2}{p}-1.$
	A very special feature of $\mathcal{P}_\lambda(x,n)$ is that it is constant on the hypersurfaces $H_{n, a_t}=\{n \sigma(a_t n'): \, n' \in N\}.$ Here $\sigma$ is the geodesic inversion on $S,$ that is an involutive, measure-preserving, diffeomorphism which can be explicitly given by \cite{CDK91}:
	\begin{align*}
	    \sigma(X,Z, a_t) = \left( \left(e^t+\frac{|V|^2}{4} \right)^2+|Z|^2 \right)^{-1} \left( \left(- \left( e^t+\frac{|X|^2}{4} \right)+J_Z \right)X, -Z, a_t  \right).
	\end{align*}

	Let $\Delta_S$ be the Laplace-Beltrami operator on $S.$ Then for every fixed $n \in N,$ $\mathcal{P}_\lambda(x, n)$ is an eigenfunction of $\Delta_S$ with eigenvalue $-(\lambda^2+\frac{Q^2}{4})$ (see \cite{ACB97}).
	For a measurable function $f$ on $S,$ the (Helgason) Fourier transform is defined as
	$$\widetilde{f}(\lambda, n)= \int_S f(x)\, \mathcal{P}_\lambda(x, n) dx$$ whenever the integral converge. For $f \in C_c^\infty(S),$ the following inversion formula holds (\cite[Theorem 4.4]{ACB97}):
	$$f(x)= C \int_{\mathbb{R}} \int_N \widetilde{f}(\lambda, n)\,\mathcal{P}_{-\lambda}(\lambda, n) |c(\lambda)|^{-2} \, d\lambda dn,$$ where $C=\frac{c_{m,l}}{2 \pi}$. The authors also proved that the (Helgason) Fourier transform extends to an isometry from $L^2(S)$ onto the space $L^2(\mathbb{R}_+ \times N, C|c(\lambda)|^{-2}d\lambda dn).$ In fact they have the precise value of constants, we refer the reader to \cite{ACB97}. The following estimates for the function $|c(\lambda)|$ holds:
	$c_1 |\lambda|^{d-1} \leq |c(\lambda)|^{-2} \leq (1+|\lambda|)^{d-1} $ for all $\lambda \in \mathbb{R}$ (e. g. see \cite{RS}).
	In \cite[Theorem 4.6]{RS}, the authors proved the following version of the Hausdorff-Young inequality:
	For $1 \leq p \leq 2$ we have 
	\begin{align}
	    \left(\int_{\mathbb{R}} \int_{N} |\widetilde{f}(\lambda+i \gamma_{p'} \rho, n)|^{p'} dn\, |c(\lambda)|^{-2} d\lambda \right)^{\frac{1}{p'}} \leq 
	C_p \|f\|_p.
	\end{align}
	  A function $f$ on $S$ is called {\it radial} if for all $x, y \in S ,$ $f(x)=f(y)$ if $\mu(x,e)=\mu(y,e),$ where $\mu$ is the metric induced by the canonical  left invariant Riemannian structure on $S$ and  $e$ is the identity element of $S.$  Note that radial functions on $S$ can be identified with the functions $f=f(r)$ of the geodesic distance  $r=\mu(x, e) \in [0, \infty)$ to the identity.  It is clear that $\mu(a_t, e)=|t|$ for $t \in \mathbb{R}.$ At times, for any radial function $f$ we use the notation $f(a_t)=f(t).$ For any function space $\mathcal{F}(S)$ on $S$, the subspace of radial functions will be denoted by $\mathcal{F}(S)^\#.$ 
	The elementary spherical function $\phi_\lambda(x)$ is defined by 
	 $$\phi_\lambda(x) :=\int_N \mathcal{P}_\lambda(x, n) \mathcal{P}_{-\lambda}(x, n)\, dn.  $$
	 It follows (\cite{Anker96, ACB97}) that $\phi_\lambda$ is a radial eigenfunction of the Laplace-Beltrami operator $\Delta_S$ of $S$ with eigenvalue $-(\lambda^2+\frac{Q^2}{4})$ such that $\phi_\lambda(x)=\phi_{-\lambda}(x),\,\, \phi_\lambda(x)=\phi_\lambda(x^{-1})$ and $\phi_\lambda(e)=1.$ It is also evident from the fact that, for every fixed $n \in N,$ $\mathcal{P}_\lambda(x, n)$ is an eigenfunction of $\Delta_S$ with eigenvalue $-(\lambda^2+\frac{Q^2}{4})$, that, for suitable function $f$ on $S,$ we have $$\widetilde{\Delta_S^l f}(\lambda, n)= -(\lambda^2+ \frac{Q^2}{4})^l \widetilde{f}(\lambda, n)$$ for every natural number $l$ (see \cite[p. 416]{ACB97}).
	 In \cite{Anker96}, the authors showed that the radial part (in geodesic polar coordinates) of the Laplace-Beltrami operator $\Delta_S$ given by 
	 $$\textnormal{rad}\, \Delta_S= \frac{\partial^2}{\partial t}+\{ \frac{m+l}{2} \coth{ \frac{t}{2}}+\frac{k}{2} \tanh{\frac{t}{2}} \} \frac{\partial}{\partial t},$$
	 is (by substituting $r=\frac{t}{2}$) equal to  $\frac{1}{4} \mathcal{L}_{\alpha, \beta}$  with indices $\alpha=\frac{m+l+1}{2}$ and $\beta=\frac{l-1}{2},$ where $\mathcal{L}_{\alpha, \beta}$ is the Jacobi operator studied by Koornwinder \cite{Koorn} in detail. It is worth noting that we are in the ideal situation of Jacobi analysis with $\alpha>\beta>\frac{-1}{2}.$ In fact, the Jacobi functions $\phi_\lambda^{\alpha, \beta}$ and elementary spherical functions $\phi_\lambda$ are related as  (\cite{Anker96}): $\phi_\lambda(t)=\phi_{2 \lambda}^{\alpha, \beta}(\frac{t}{2}).$ 
	As consequence of this relation, the following estimates for the elementary spherical functions hold true (see \cite{Platonov}).
	\begin{lem} \label{estijac}  The following inequalities are valid for spherical functions $\phi_\lambda(t)\,\,(t, \lambda \in \mathbb{R}_+):$
	\begin{itemize}
	    \item $|\phi_\lambda(t)| \leq 1.$
	    \item $|1-\phi_\lambda(t)| \leq \frac{t^2}{2} ( \lambda^2+\frac{Q^2}{4}).$
	    \item There exists a constant $c>0,$ depending only on $\lambda,$ such that $|1-\phi_\lambda(t)| \geq c$ for  $\lambda t \geq 1.$ 
	\end{itemize}
	\end{lem} 
%	Now, we define the spherical (or Harish Chandra) transform of an integrable radial function $f$ on $S$ by 
%	$$\mathcal{H}(f)(\lambda):= \int_S f(x) \phi_\lambda(x)\, dx.$$
%	For $f \in C_c^\infty(S)^\#,$ the following inversion formula holds:
%	$$\mathcal{H}^{-1}(f)(x):= c_S \int_{0}^\infty \mathcal{H}f(\lambda) \phi_\lambda(x)\,|c(\lambda)|^{-2}\,d\lambda,$$ where $c_S$ depends only on $m$ and $l,$ and $c(\lambda)$ is the Harish-Chandra function. 
%	Moreover, the following Plancherel formula holds:
%	   $$\int_S |f(x)|^2 dx = c_S\int_{\mathbb{R}_+} |\mathcal{H}f(\lambda)|^2\, |c(\lambda)|^{-2} d\lambda.$$
%	   The spherical Fourier transform extends to an isometry from $L^2(S)^\#$ to $L^2(\mathbb{R}_+, c_S |c(\lambda)|^{-2}d\lambda).$ 
	   
	   Let $\sigma_t$ be the normalized surface measure of the geodesic sphere of radius $t$. Then $\sigma_t$ is a nonnegative radial measure. The spherical mean operator $M_t$ on a suitable function space on $S$ is defined by $M_tf: =f*\sigma_t.$  It can be noted that $M_tf(x)=\mathcal{R}(f^x)(t)$, where $f^x$ denotes  the right translation of function $f$ by $x$ and $\mathcal{R}$ is the radialization operator defined, for suitable function $f,$ by
	   $$\mathcal{R}f(x)=\int_{S_\nu} f(y)\,d\sigma_{\nu}(y),$$ where $\nu=r(x)= \mu(C(x), 0),$ here $C$ is the Cayley transform, and $d\sigma_\nu$ is the normalized surface measure induced by the left invariant Riemannian metric on the geodesic sphere $S_\nu=\{y \in S: \mu(y, e)=\nu \}.$ It is easy to see that  $\mathcal{R}f$ is a radial function and for any radial function $f,$ $\mathcal{R}f=f.$ Consequently, for a radial function $f,$\, $M_tf$ is the usual translation of $f$ by $t.$  In \cite{KRS}, the authors proved that, for a suitable function $f$ on $S,$ $\widetilde{M_t f}(\lambda, n)= \widetilde{f}(\lambda, n) \phi_\lambda(t)$ whenever both make sense. Also, $M_t f$ converges to $f$ as $t \rightarrow 0,$ i.e., $\mu(a_t, e) \rightarrow 0.$ It is also known that $M_t$ is a bounded operator on $L^2(S)$ with operator norm equal to $\phi_0(a_t).$ In particular, for $f \in L^2(S),$ we have   $\|M_t f\|_{2} \leq \phi_0(a_t) \|f\|_2.$
	  The following Lemmata are taken from \cite{BrayPinsky}.
	  \begin{lem} Let $\alpha>\frac{-1}{2}.$ Then there are positive constant $c_{1, \alpha}$ and $c_{2, \alpha}$ such that 
	  $$ c_{1, \alpha} \min \{1, (\lambda t)^2\} \leq 1- j_\alpha(\lambda t) \leq c_{2, \alpha} \min \{1, (\lambda t)^2 \},$$ where $j_\alpha$ is the usual Bessel function of first kind normalized by $j_\alpha(0)=1.$
	  \end{lem}
	   \begin{lem} \label{Lembes} Let $\alpha >\frac{-1}{2}$ and $t_0 >0.$ Then, for all $\lambda \in \mathbb{R}$, there exist a constant $c_1>0$ such that for all $0 \leq t \leq t_0,$ the function $\phi_\lambda$ satisfies 
	   $$|1- \phi_\lambda(t)| \geq c_1 |1-j_\alpha(\lambda t) |,$$  where $j_\alpha$ is the usual Bessel function of first kind normalized by $j_\alpha(0)=1.$
	   \end{lem}
	   
	   \section{Main results}
	 In this section we present our main results. Throughout this section, we denote a Damek-Ricci space by $S$. We denote by $L^2(S)$ the Hilbert space of all square integrable function on $S$ with respect to Haar measure $\lambda$ on $S.$  We begin this section by recalling the definition of Sobolev spaces on Damek-Ricci spaces. 
	   
The Sobolev space $W_2^m(S)$ on Damek-Ricci space $S$ is defined by 
$$W_2^m(S):= \{ f \in L^2(S): \Delta_S^l f \in L^2(S), \quad l=1,2,\ldots, m \}.$$
The space $W_2^m(S)$ can be equipped with seminorm $|f|_{W_2^m(S)}:= \|\Delta_S^m f\|_{2}$ and with the norm $\|f\|_{W_2^m(S)} = \|f\|_2+ \|\Delta_S^m f\|_{2}.$

The modulus of smoothness (continuity) $\Omega_k$ is defined by using the spherical mean operator $M_t$ as follows:
$$\Omega_k(f, \delta)_2:= \sup_{0<t \leq \delta} \|\Delta_t^k f\|_2,$$ where $\Delta_t^k f= (I-M_t)^k f.$
The modulus of smoothness $\Omega_k(f, \delta)_2$ satisfies the following properties:
\begin{itemize}
    \item[(i)] The function $\delta \mapsto \Omega_k(f, \delta)_2$ is a decreasing function and satisfies  $$\Omega_k(f \pm g, \delta)_2 \leq \Omega_k(f, \delta)_2+\Omega_k(g, \delta)_2$$ for all $f, g \in L^2(S).$
    \item[(ii)] $\Omega_k(f, \delta)_2 \leq (\phi_0(a_t)+1)^k \|f\|_2$ and $\Omega_k(f, \delta)_2 \leq (1+\phi_0(a_t))^{k-l} \Omega_l(f, \delta)_2$ for $l \leq k.$ 
    \item [(iii)] If $f \in W^m_2(S)$ then we have $\Omega_k(f, \delta)_2 \leq \delta^{2k} \|\Delta_S^k f\|_2, \quad k \leq m.$ 
\end{itemize}
The proof of (i) and (ii) follows from the definition of modulus of continuity and norm estimate for $M_t$ on $L^2(S).$ To show (iii), we note, by Plancherel formula, that,
\begin{align*}
    \|\Delta_t^k f\|_2^2= \int_0^\infty \int_N |\widetilde{(\Delta_t^k f)}(\lambda, n)|^2\, |c(\lambda)|^{-2} d\lambda\, dn.  
\end{align*}
Since $\widetilde{(\Delta_t^k f)}(\lambda, n)= |1-\phi_\lambda(a_t)|^k \widetilde{f}(\lambda, n)$ we have, by Lemma \ref{estijac}, that
\begin{align*}
        \|\Delta_t^k f\|_2^2&= \int_0^\infty \int_N |1-\phi_\lambda(a_t)|^{2k} |\widetilde{f}(\lambda, n)|\, |c(\lambda)|^{-2} d\lambda\, dn  \\& \leq t^{4k} \int_0^\infty \int_N (\lambda^2+\frac{Q^2}{4})^{2k} |\widetilde{f}(\lambda, n)|\, |c(\lambda)|^{-2} d\lambda\, dn \\&= t^{4k} \int_0^\infty \int_N |\widetilde{\Delta_S^k f}(\lambda, n)|^2\,  |c(\lambda)|^{-2} d\lambda\, dn = t^{4k} \|\Delta_S^k f\|_{L^2(S)}^2.
\end{align*}
\subsection{ Direct Jackson theorem} This subsection is devoted for proving the Direct Jackson theorem of approximations theory for Damek-Ricci spaces. For the approximation we will use the functions of bounded spectrum. The functions of bounded spectrum were used by Platonov \cite{Platov9, Platov8, Platov3, Platov5} to prove Jackson type direct theorem for Jacobi transform and for symmetric spaces. Such kind of functions also appear in the work of Pesenson \cite{Pens3} under the name of Paley-Wiener functions for studying approximation theory on homogeneous manifolds.  

A function $f \in L^2(S)$ is called  a {\it function with bounded spectrum} (or a {\it Paley-Wiener function}) of order $\nu>0$ if $$\mathcal{F}f(\lambda, n)=0\quad \text{for}\,\,|\lambda|>\nu.$$ Denote the space of all function on $S$ with bounded spectrum of order $\nu$  by $\text{BS}_\nu(S).$ 
The best approximation of a function $f \in L^2(S)$ by the functions in $\text{BS}_\nu(S)$ is defined by $$E_\nu(f):=\inf_{g \in \text{BV}_\nu(S)}\|f-g\|_{L^2(S)}.$$

\begin{lem} \label{3.1lema}
    Let $\nu>0.$ For any function $f \in L^2(S),$ the function $P_\nu(f)$ defined by 
    $$P_{\nu}(f)(x):=\mathcal{F}^{-1}(\mathcal{F}f(\lambda, n) \chi_\nu(\lambda)),$$ where $\chi_\nu$ is a function defined by $\chi_\nu(\lambda)=1$ for $|\lambda|\leq \nu$ and $0$ otherwise, satisfies the following properties:
    \begin{itemize}
        \item [(i)] For every $f \in L^2(S),$ $P_\nu(f) \in \text{BS}_\nu(S).$
        \item[(ii)] For every function $f \in \text{BS}_\nu(S),$ $P_\nu(f)=f.$
        \item[(iii)] If $f \in L^2(S)$ then $\|P_\nu(f)\|_{L^2(S)}\leq \|f\|_{L^2(S)}$ and $\|f-P_\nu(f)\|_{L^2(S)} \leq 4 E_\nu(f).$   \end{itemize}
\end{lem}
\begin{proof}
    \begin{itemize}
        \item[(i)] This is trivial to see. Indeed, by definition we have  $$\mathcal{F} P_{\nu}(f)(x)=\mathcal{F}f(\lambda, n) \chi_\nu(\lambda)=0$$ for $|\lambda|>\nu.$ Therefore, $P_\nu(f) \in BS_\nu(S).$
        \item[(ii)] Let $f \in BS_\nu(S).$ Then $\mathcal{F}f(\lambda, n)=0$ for $ |\lambda|>\nu$ and $\mathcal{F}P_\nu(f)(\lambda, n)=\mathcal{F}f(\lambda, n)$ for $|\lambda| \leq \nu.$ So, by using the inversion formula we have 
        \begin{align*}
            P_\nu(f)(x)&=C\int_{\mathbb{R}} \int_N \mathcal{F}P_\nu(f)(\lambda, n)\,\, |c(\lambda)|^{-2} d\lambda\, dn\\&= C \int_{|\lambda| \leq \nu} \int_N \mathcal{F}f(\lambda, n)\,\, |c(\lambda)|^{-2} d\lambda\, dn \\&= C \int_{\mathbb{R}} \int_N \mathcal{F}f(\lambda, n)\,\, |c(\lambda)|^{-2} d\lambda\, dn= f(x).
        \end{align*}
        \item[(iii)] Take $f \in L^2(S).$ By Plancherel formula, we get 
        \begin{align*}
            \|P_\nu(f)\|_{L^2(S)}^2 &= \int_{0}^\infty \int_N |\mathcal{F}P_\nu(f)(\lambda, n)|^2\, |c(\lambda)|^{-2}\,d\lambda\,dn \\&=\int_{0}^\nu \int_N |\mathcal{F}f(\lambda, n)|^2\, |c(\lambda)|^{-2}\,d\lambda\,dn \\&\leq  \int_{0}^\infty \int_N |\mathcal{F}f(\lambda, n)|^2\, |c(\lambda)|^{-2}\,d\lambda\,dn = \|f\|_{L^2(S)}^2.
        \end{align*}
        Also, for proving second inequality take any $g\in BS_\nu(S)$ 
        such that 
        $$\|f-g\| \leq 2 E_\nu(f)_2.$$
        Now, by using the fact that $P_\nu(g)=g$ we get
        \begin{align*}
            \|f-P_\nu(f)\|_{L^2(S)}&=\|f-g-P_{\nu}(g-f)\|_{L^2(S)} \\&\leq \|f-g\|_{L^2(S)}+\|f-g\|_{L^2(S)} \leq 4 E_\nu(f)_2.
        \end{align*}\end{itemize}\end{proof}

The following two theorems are analogues of Jackson's direct theorem in classical approximation theorem for Damek-Ricci spaces.  

\begin{thm} \label{vishthm3.2}
    If $f \in L^2(S)$ then for every $\nu>0$ we have 
    \begin{equation} \label{2vish}
        E_\nu(f) \leq c_k \,\,\Omega_k\left(f, \frac{1}{\nu}\right)_2, \quad k \in \mathbb{N},
    \end{equation} where $c_k$ is a constant. 
\end{thm}
\begin{proof} The Plancherel formula gives that 
\begin{align*}
    \|f-P_\nu(f)\|_{L^2(S)}^2 &= \int_0^\infty \int_N |\mathcal{F}(f-P_\nu(f))(\lambda, n)|^2\,\, |c(\lambda)|^{-2}\, d\lambda\, dn \\&= \int_0^\infty \int_N |1-\chi_{\nu}(\lambda)|^2\,|\mathcal{F}(f(\lambda, n)|^2\,\, |c(\lambda)|^{-2}\, d\lambda\, dn \\&= \int_{\lambda \geq \nu}\int_N |\mathcal{F}(f(\lambda, n)|^2\,\, |c(\lambda)|^{-2}\, d\lambda\, dn. 
\end{align*}
By Lemma \ref{estijac} we have $|1-\phi_\lambda\left(\frac{1}{\nu} \right)| \geq c$ for $\lambda \geq \nu.$ Therefore, by Plancherel formula, we get 
\begin{align*}
    \|f-P_\nu(f)\|_{L^2(S)}^2 &\leq c^{-2k} \int_{\lambda \geq \nu} \int_N |1-\phi_\lambda\left(1/\nu \right)|^{2k} |\mathcal{F}f(\lambda, n)|^{2}\,|c(\lambda)|^{-2}\,d\lambda\, dn \\&= c^{-2k} \int_{\lambda \geq \nu} \int_N |\mathcal{F}((I-M_{1/\nu})^k f)(\lambda, n)|^2 \,|c(\lambda)|^{-2}\,d\lambda\, dn \\&\leq c^{-2k} \int_0^\infty \int_N |\mathcal{F}((I-M_{1/\nu})^k f)(\lambda, n)|^2 \,|c(\lambda)|^{-2}\,d\lambda\, dn \\&= c^{-2k} \|(I-M_{1/\nu})^k f\|_{L^2(S)}^2.
\end{align*}
Therefore, as $P_\nu(f) \in BS_\nu(S),$ we get 
    \begin{align*}
       E_\nu(f)= \inf_{g \in \text{BV}_\nu(S)}\|f-g\|_{L^2(S)} &\leq  \|f-P_\nu(f)\|_{L^2(S)} \leq c^{-k} \|(I-M_{1/\nu})^k f\|_{L^2(S)}\\&= c^{-k} \|\Delta_{1/\nu}^k f\|_{L^2(S)} \leq  c_k\,\, \Omega_k\left(f, \frac{1}{\nu} \right)_2,
    \end{align*}
    proving \eqref{2vish} and hence the theorem is proved. 
 \end{proof}
 \begin{thm}
     Let $r \in \mathbb{N}$ and $\nu>0.$ Assume that $f, \Delta_Sf, \Delta^2f, \ldots, \Delta^rf$ are in $L^2(S).$ Then 
     \begin{equation} \label{3vish}
        E_\nu(f) \leq c_k' \,\, \nu^{-2r}\Omega_k\left(\Delta_S^rf, \frac{1}{\nu}\right)_2, \quad k \in \mathbb{N},
    \end{equation} where $c_k'$ is a constant.
 \end{thm}
 \begin{proof} Let $r \in \mathbb{N}$ and $t>0.$ Suppose that $f, \Delta_Sf, \Delta^2_Sf, \ldots, \Delta^r_Sf$ are in $L^2(S).$ Then Lemma \ref{estijac} and Plancherel formula give that 
 \begin{align*}
     \|(I-M_{t})f\|^2_{L^2(S)}&= \int_0^\infty \int_N |\mathcal{F}((I-M_{t})f)(\lambda, n)|^2 |c(\lambda)|^{-2}\, d\lambda\, dn \\&=\int_0^\infty \int_N |1-\phi_\lambda(a_t)|^2 |\mathcal{F}f(\lambda, n)|^2 |c(\lambda)|^{-2}\, d\lambda\, dn \\&\leq \frac{t^4}{4}  \int_0^\infty \int_N (\lambda^2+\frac{Q^2}{4})^2 |\mathcal{F}f(\lambda, n)|^2 |c(\lambda)|^{-2}\, d\lambda\, dn \\&=\frac{t^4}{4}  \int_0^\infty \int_N  |\mathcal{F}(\Delta_S f)(\lambda, n)|^2 |c(\lambda)|^{-2}\, d\lambda\, dn = \frac{t^4}{4} \|\Delta_S f\|^2_{L^2(S)}.
 \end{align*}
 Therefore, 
 \begin{equation} \label{vish(4)}
     \|(I-M_{t})f\|_{L^2(S)} \leq \frac{t^2}{2} \|\Delta_S f\|_{L^2(S)}.
 \end{equation}
 By proceeding similar to the proof of Theorem \ref{vishthm3.2} we get 
 \begin{equation} \label{vish(5)}
     \|f-P_\nu(f)\|_{L^2(S)} \leq c^{-(k+r)} \|(I-M_{1/\nu})^{k+r} f\|_{L^2(S)}.
 \end{equation}
 By applying inequality \eqref{vish(4)} on the right hand side of \eqref{vish(5)} $r$-times we obtain that 
 \begin{align*}
     \|f-P_\nu(f)\|_{L^2(S)} &\leq c^{-(k+r)} 2^{-r} \nu^{-2r} \|(I-M_{1/\nu})^{k} \Delta_S^rf\|_{L^2(S)} \\&= c_k'\,\, \nu^{-2r} \Omega_k \left(\Delta_S^rf, \frac{1}{\nu} \right)_2,
 \end{align*} where $c_k'=c^{-(k+r)} 2^{-r}.$ Now, the theorem follows from the definition of $E_\nu(f)$ by noting that
 \begin{align*}
     E_\nu(f)= \inf_{g \in \text{BV}_\nu(S)}\|f-g\|_{L^2(S)} \leq \|f-P_\nu(f)\|_{L^2(S)} \leq c_k'\,\, \nu^{-2r} \Omega_k \left(\Delta_S^rf, \frac{1}{\nu} \right)_2,  
 \end{align*} completing the proof.

 \end{proof}

 \subsection{Nikolskii-Stechkin inequality} In this subsection, we will prove Nikolskii-Stechkin inequality \cite{Niko} for Damek-Ricci spaces.  
 \begin{thm}
     For any $f \in L^2(S)$ and $\nu >0$ we have 
     \begin{equation}
          \|\Delta_S^k(P_\nu(f))\|_{L^2(S)} \leq c_3\, \nu^{2k} \|\Delta_{1/ \nu}^{k} f\|_{L^2(S)}, \quad k \in \mathbb{N}.
     \end{equation}
     %Also, if $f \in W_2^k(S)$ then we have 
     %\begin{equation}
      %   \|\Delta_{1/\nu}^k (P_\nu(f))\|_{L^2(S)} \leq \nu^{-2k}\,\|\Delta_S^k\|_{L^2(S)}, \quad k \in \mathbb{N}.
     %\end{equation}
 \end{thm}
 \begin{proof}
 First note that  
 \begin{align*}
     \mathcal{F}(\Delta_S^k P_\nu(f))(\lambda, n)= (-1)^k \left(\lambda^2+\frac{Q^2}{4}\right)^{k} \mathcal{F}(P_\nu(f))(\lambda, n).
 \end{align*}

 Using Plancherel formula we have 
 \begin{align*}
     \|\Delta_S^k(P_\nu(f))\|_{L^2(S)}^2 &= \int_0^\infty \int_N |\mathcal{F}(\Delta_S^k P_\nu(f))(\lambda, n)|^2 |c(\lambda)|^{-2}\, d\lambda\, dn \\&= \int_{|\lambda| \leq \nu} \int_N \left(\lambda^2+Q^2/4\right)^{2k} |\mathcal{F}f(\lambda, n)|^2\, |c(\lambda)|^{-2}\, d\lambda\, dn \\&= \int_0^\infty \int_N \frac{\left(\lambda^2+Q^2/4\right)^{2k} \chi_\nu(\lambda)}{|1-\phi_\lambda(1/\nu)|^{2k}} |1-\phi_\lambda(1/\nu)|^{2k} |\mathcal{F}f(\lambda, n)|^2\, |c(\lambda)|^{-2}\, d\lambda\, dn.
 \end{align*}
 Now note that by Lemma \ref{Lembes} we have
 \begin{align*}
     \sup_{\lambda \in \mathbb{R}} \frac{\left(\lambda^2+Q^2/4\right)^{2k} \chi_\nu(\lambda)}{|1-\phi_\lambda(1/\nu)|^{2k}}&=  \nu^{4k}\sup_{|\lambda| \leq \nu} \frac{\left((\lambda^2+Q^2/4)/ \nu^2\right)^{2k} }{|1-\phi_\lambda(1/\nu)|^{2k}} \\&\leq \frac{\nu^{4k}}{c_1} \sup_{|\lambda| \leq \nu} \frac{\left((\lambda^2+Q^2/4)/ \nu^2\right)^{2k} }{|1-j_\alpha(\lambda/\nu)|^{2k}} \\&= \frac{\nu^{4k}}{c_1} \sup_{|t| \leq 1} \frac{\left(t^2+Q^2/4\nu^2\right)^{2k} }{|1-j_\alpha(t)|^{2k}} = \frac{C'}{c_1} \nu^{4k}, 
 \end{align*} where $C'= \sup_{|t| \leq 1} \frac{\left(t^2+Q^2/4\nu^2\right)^{2k} }{|1-j_\alpha(t)|^{2k}}.$
 
 Therefore, we get 
 \begin{align*}
     \|\Delta_S^k(P_\nu(f))\|_{L^2(S)}^2 &\leq  \frac{C'}{c_1} \nu^{4k}\int_0^\infty \int_N  |1-\phi_\lambda(1/\nu)|^{2k} |\mathcal{F}f(\lambda, n)|^2\, |c(\lambda)|^{-2}\, d\lambda\, dn \\&= \frac{C'}{c_1} \nu^{4k}\int_0^\infty \int_N   |\mathcal{F}(\Delta_{1/\nu}^{k}f)(\lambda, n)|^2\, |c(\lambda)|^{-2}\, d\lambda\, dn \\&= \frac{C'}{c_1} \nu^{4k} \|\Delta_{1/\nu}^{k}f\|^2_{L^2(S)}.
 \end{align*}
 Hence, $\|\Delta_S^k(P_\nu(f))\|_{L^2(S)} \leq c_3\, \nu^{2k}\,\|\Delta_{1/\nu}^{k}f\|_{L^2(S)}.$
 \end{proof}
 As noted in  Lemma \ref{3.1lema} that $P_\nu(f)=f$ for any $f \in BS_\nu(S),$ the following corollary is immediate. 
 \begin{cor} For $\nu>0,\, k \in \mathbb{N}$ and $f \in BV_\nu(S)$ we have the following inequality:
 $$\|\Delta_S^k f\|_{L^2(S)} \leq c_3 \, \nu^{2k}\,\|\Delta_{1/\nu}^{k}f\|_{L^2(S)}.$$
 \end{cor}
 The following corollary follows from the definition of modulus of smoothness.
 \begin{cor} \label{3.5coro} For $\nu>0,\, k \in \mathbb{N}$ and $f \in L^2(S)$ we have the following inequality:
 $$\|\Delta_S^k f\|_{L^2(S)} \leq c_3\, \nu^{2k}\,\Omega_k\left(f, \frac{1}{\nu} \right)_2.$$
 \end{cor}

 \subsection{Equivalence of the $K$-functional and modulus of smoothness} Our main objective will be proved here. We will prove in the following theorem that the $K$-functional for the pair $(L^2(S), W_2^m(S))$ and modulus of smoothness generated by spherical mean operators are equivalent. The Peetre the $K$-functional $K(f, \delta, L^2(S), W_2^m(S))$ for the pair $(L^2(S), W_2^m(S))$ is defined by 
 $$K_m(f, \delta):= \inf\{\|f-g\|_{L^2(S)}+\delta \|\Delta_S^m g\|_{L^2(S)}:\,\,f \in L^2(S)\,\, g \in W_2^m(S)\}.$$

The next theorem presents the equivalence of the $K$-functional $K_m(f, \delta^{2m})$ and the modulus of smoothness $\Omega_m(f, \delta)_2$ for $f \in L^2(S)$ and $\delta>0.$

 \begin{thm} \label{main}
     For $f \in L^2(S)$ and $\delta>0$ we have 
     \begin{equation} \label{eq4vish}
          \Omega_m(f, \delta)_2 \asymp K_m(f, \delta^{2m}).
     \end{equation}
     In other words, there exist $c_1>0,\,\,c_2>0$ such that for all $f \in L^2(S)$ and $\delta>0$ we have 
     $$c_1\, \Omega_m(f, \delta)_2 \leq  K_m(f, \delta^{2m})\leq c_2\, \Omega_m(f, \delta)_2.$$
     \end{thm}
 \begin{proof} Take $g \in W_2^m(S).$ Now by using the properties of modulus of continuity $\Omega_m(f, \delta)_2$ we get
 \begin{align*}
     \Omega_m(f, \delta)_2 &\leq \Omega_m(f-g, \delta)_2+\Omega_m(g, \delta)_2 \\&\leq (\phi_0(a_t)+1)^m \|f-g\|_{L^2(S)}+\delta^{2m} \|\Delta_S^m g\|_{L^2(S)}\\& \leq \tilde{c} (\|f-g\|_{L^2(S)}+\delta^{2m} \|\Delta_S^m g\|_{L^2(S)}),
 \end{align*} where $\tilde{c}=  (\phi_0(a_t)+1)^m.$ By taking the infimum over all $g \in W_2^m(S),$ we obtain 
 $$\Omega_m(f, \delta)_2 \lesssim K_m(f, \delta^{2m}).$$
 Now, to prove the other side we take $g=P_\nu(f)$ for $\nu>0,$ then, from the definition of $K_m(f, \delta^{2m}),$ it follows that 
 \begin{equation}
     K_m(f, \delta^{2m}) \leq \|f-P_\nu(f)\|_{L^2(S)}+\delta^{2m} \|\Delta_S^m(P_\nu(f))\|_{L^2(S)}.
 \end{equation}
 Now, from Lemma \ref{3.1lema} (iii), \eqref{2vish} and Corollary \ref{3.5coro} we get that 
 \begin{align*}
     K_m(f, \delta^{2m}) &\leq 4 E_v(f)+ c_3\delta^{2m} \nu^{2m} \Omega_m\left(f, \frac{1}{\nu}\right)_2 \\& \leq 4c_2 \Omega_m\left(f, \frac{1}{\nu}\right)_2 + c_3(\delta \nu)^{2m} \Omega_m\left(f, \frac{1}{\nu}\right)_2 \leq c_4 (1+(\delta \nu)^{2m}) \Omega_m\left(f, \frac{1}{\nu}\right)_2.
 \end{align*}  By taking $\nu=\frac{1}{\delta}$ we get 
 $$K_m(f, \delta^{2m}) \lesssim \Omega_m\left(f, \delta \right)_2$$ proving \eqref{eq4vish}. \end{proof}

	\section*{Acknowledgment}
	VK and MR are supported by FWO Odysseus 1 grant G.0H94.18N: Analysis and Partial Differential Equations. MR is also supported by the Leverhulme Grant RPG-2017-151 and by EPSRC Grant EP/R003025/1.

\bibliographystyle{amsplain}

	\end{document}